\documentclass[leqno,twoside,12pt]{amsart}
\usepackage{latexsym, esint, url, amssymb, yfonts, amsmath}
\usepackage[utf8]{inputenc}
\usepackage{graphicx,color}
\usepackage{tikz-cd}
\usepackage{caption}
\usepackage{subcaption}
\usepackage{array}
\newtheorem*{theorem*}{Theorem} 
\newtheorem{theorem}{Theorem}[section]
\newtheorem{lemma}[theorem]{Lemma}

\newtheorem{corollary}[theorem]{Corollary}
\newtheorem{definition}[theorem]{Definition}

\theoremstyle{definition}
\newtheorem{remark}[theorem]{Remark}
\newtheorem{example}[theorem]{Example}

\numberwithin{equation}{section}
\numberwithin{figure}{section}

\definecolor{Green}{rgb}{0, 0.65,0}

\begin{document}
		
\title[On the normalised $p$-parabolic equation]{On the normalised $p$-parabolic equation in arbitrary domains}
\author{Nikolai Ubostad}
\address{Department of Mathematical Sciences 
Norwegian University of Science and Technology 
N-7491 Trondheim, Norway}
\email{nikolai.ubostad@ntnu.no}
\date{\today}

\begin{abstract}
The boundary regularity for the normalised \newline $p$-parabolic equation $u_t =\frac{1}{p}|Du|^{2-p}\Delta_pu$ is studied. Perron's method is used to construct solutions in arbitrary domains. We classify the regular boundary points in terms of barrier functions, and prove an Exterior Sphere condition. We identify a fundamental solution, and a Petrovsky criterion is established.  We examine the convergence of solutions as $p \to \infty$.
\end{abstract}
\maketitle
\section{Introduction}
\noindent We investigate Perron solutions of the \emph{normalised} $p$-parabolic equation
\begin{equation}
\label{eq:evoplap}
\begin{split}
u_t &=\frac{1}{p}|Du|^{2-p}\Delta_pu \\
&=\frac{1}{p}\text{tr}(D^2u) +\frac{p-2}{p}\left\langle D^2u\frac{Du}{|Du|}, \frac{Du}{|Du|} \right\rangle,
\end{split}
\end{equation}
where $1<p <\infty$, in general domains $\Omega \subset \mathbb{R}^n\times (-\infty, \infty)$. Here $\Delta_pu$ denotes the $p$-Laplace operator of $u$,
\[
\Delta_pu=\text{div}(|Du|^{p-2}Du),
\]
and we let
\[
\mathcal{A}_pu:=\frac{1}{p}|Du|^{2-p}\Delta_pu.
\]
The operator $\mathcal{A}_p$ is called normalised  since it is homogeneous of degree 1, that is $\mathcal{A}_p(\alpha u)= \alpha \mathcal{A}_pu$. In contrast, the $p$-Laplace operator is homogeneous of degree $p-1$.
 
 Since the equation cannot be written on divergence form, the distributional weak solutions are not available to us. The correct notion is the viscosity solutions, introduced in \cite{crandall1983viscosity}.
 
 Sternberg \cite{sternberg1929gleichung} observed that Perron's method for solving the Dirichlet boundary value problem for Laplace's equation in \cite{perron1923neue} could be extended to the heat equation. We adapt Perron's method to the non-linear equation \eqref{eq:evoplap} in general domains, not necessarily space-time cylinders.
 
Equation \eqref{eq:evoplap}   was studied by Does in connection with  image processing, see \cite{does2009evolution}.  The existence of viscosity solutions on cylinders  $Q_T=Q \times (0, T )$ was established by Perron's method.

Integral to the potential theory is the regularity of boundary points. A point $\zeta_0\in \partial \Omega$ is called regular if, for \emph{every} continuous function $f: \partial \Omega \to \mathbb{R}$, we have
\[
\lim_{\eta \to \zeta_0}u(\eta)=f(\zeta_0), \ \eta \in \Omega.
\]
The boundary values are prescribed as for an elliptic problem. For example, the points on $Q \times \{t=T\}$ are not regular boundary points of the cylinder $Q_T$.

We characterize the regularity of boundary points in terms of \emph{barriers}. We say that $w$ is a barrier at $\zeta$ if $w$ is a positive supersolution of \eqref{eq:evoplap} defined on the entire domain, such that $w(\zeta)=0$ and $w(\eta)>0$ for $\zeta \neq \eta\in \partial \Omega$.

To keep the presentation within reasonable limits, we investigate only the case where the boundary data $f$ is bounded. This is not a serious restriction.

The related equation
 \begin{equation}
 \label{eq:banerjeeeq}
 u_t=|Du|^{2-p}\Delta_pu,
 \end{equation}
 without the factor $1/p$ present, was investigated in \cite{banerjee2015dirichlet}. 
For $p \geq 2$, it was proved that a boundary point $\zeta$ is regular if, and only if, there exists a barrier at $\zeta$. The authors also showed that in the case of space - time cylinder $Q_T$, $(x, t) \in \partial Q\times (0, T]$, is a regular boundary point if and only if $x \in \partial Q$
is a a regular boundary point for the elliptic $p$-Laplacian. 

The regularity of a point is a very delicate issue. Using the Petrovsky criterion, one can construct a domain where the origin is regular for the equation $u_t =\Delta u$, while it is \emph{irregular} for $u_t =\frac{1}{2}\Delta u$, cf \cite{watson2012introduction}. Therefore it is quite remarkable that  as $p \to \infty$, the domain in our Petrovsky criterion converges precisely to the domain in the Petrovsky criterion for the normalised $\infty$-parabolic equation \eqref{eq:norminftylaplace} derived in \cite{ubostad1}, namely that the origin is an irregular point for the domain enclosed by the hypersurfaces
\begin{equation}
\label{eq:infinitypetrovsky}
\begin{split}
&\{(x, t)\in \mathbb{R}^n\times (0, \infty)\ : \ |x|^2 =-4t\log|\log|t||\} \\ &\text{and} \ \{t=-c\}, 
\end{split}
\end{equation}
for $0<c<1$.

We consider \eqref{eq:evoplap} instead of \eqref{eq:banerjeeeq} because of the convergence properties as $p \to \infty$. Results similar to ours could easily be established for \eqref{eq:banerjeeeq} by the same methods.
Indeed, this was recently proven in \cite{bjorn2017tusk}  .
The regular, non-normalised $p$-parabolic equation
\begin{equation}
\label{eq:pparabolic}
u_t = \Delta_pu,
\end{equation}
$1 < p <\infty$, has an interesting history. The initial value problem was first studied by Barenblatt in connection with the propagation of heat after thermonuclear detonations in the atmosphere,  cf. \cite{barenblatt}. The equation has several applications, for example
 in image processing  with variable $p=p(x)$ in \cite{imageenhancement}. The $p$-parabolic equation, together with its stationary counterpart the $p$-Laplace equation, also have interesting applications related to game theory and "Tug-of-War'' games, see \cite{manfredi2010asymptotic} and \cite{peres2008tug}. For the regularity theory regarding equations of this type, we mention \cite{dibenedetto1995}.

For $p=1$, the equation is connected to motion by mean curvature, investigated by Evans and Spruck in \cite{evans1991motion}. The Dirichlet problem in general domains for \eqref{eq:pparabolic} was studied by Kilpeläinen and Lindqvist in \cite{kilpelainen1996dirichlet}. See also \cite{Juutinen2006, crandall2003another}.

Also worth mentioning is the case $p=\infty$, when we get the $\infty$-parabolic equation
\begin{equation}
\label{eq:inftylaplace}
u_t=\Delta_{\infty}u,
\end{equation}
and the related \emph{normalised} $\infty$-parabolic equation
\begin{equation}
\label{eq:norminftylaplace}
u_t=|Du|^{-2}\Delta_{\infty}u := \Delta_{\infty}^Nu,
\end{equation}
where the $\infty$-Laplace operator is given by
\[
\Delta_{\infty} u =\left\langle D^2u \ Du, Du \right\rangle.
\]
 Both of these have an interesting theory in their own right, see \cite{Juutinen2006} regarding the normalised case and \cite{crandall2003another} for \eqref{eq:inftylaplace}.
 
 As $p \to \infty$, \eqref{eq:evoplap} converges to \eqref{eq:norminftylaplace} and \emph{not} to \eqref{eq:inftylaplace}.

Our first result is a characterization of regular boundary points via exterior spheres:
\begin{theorem}[Exterior Sphere]
\label{thm:exterior sphere}
Let $\zeta_0 =(t_0, x_0) \in \partial \Omega$, and suppose that there exists a closed ball $\{(x, t) \ : \ |x-x'|^2+(t-t')^2 \leq R_0^2\}$ intersecting $\overline{\Omega}$ precisely at $\zeta_0$. Then $\zeta_0$ is regular, if the intersection point is \underline{not} the south pole, that is $(x_0, t_0)  \neq (x', t'-R_0)$. If the point of intersection is the north pole, we must restrict the radius of the sphere.
\end{theorem}
We use a barrier function to prove the following Petrovsky criterion:
\begin{theorem}
\label{thm:petrovsky}
The origin $(0, 0)$ is a regular point for the domain enclosed by the hypersurfaces
\begin{equation*}
\{(x, t)\in \mathbb{R}^n\times (0, \infty)\ : \ |x|^2 =-\beta t\log|\log|t||\} \ \text{and} \ \{t=-c\}, 
\end{equation*}
for $0<c<1$, where
\[
\beta = 4 \ \frac{p-1}{p}.
\]
\end{theorem}
We also have the following irregularity result, showing that Theorem \ref{thm:petrovsky} is in some sense sharp:
\begin{theorem}
 \label{thm:not regular}
 The origin is \emph{not} a regular point of the domain $\Omega$ defined by 
 \[
 |x|^2 =-\beta(1+\epsilon)t\log|{\log{|t|}}|, \ \ t=-c
\]
  for \emph{any} $\epsilon >0$, and $\beta$ as in Theorem \ref{thm:petrovsky}.
 \end{theorem}
 These results are similar to the classical Petrovsky criterion for the heat equation, derived in \cite{petrovsky1935ersten}.
  
 The article is structured as follows. In Section \ref{sec:severalsolutions} we investigate several explicit solutions of \eqref{eq:evoplap}. We also transform it into a heat equation with variable coefficient. Basic facts regarding viscosity solutions, Perron solutions, a comparison principle and the barrier characterization are displayed in Section \ref{sec:comparison}.  The exterior sphere condition in Theorem \ref{thm:exterior sphere} is derived in Section \ref{sec:exterior sphere}. As a demonstration of the necessity of eliminating the south pole, we show that the latest moment of the $p$-parabolic ball is not regular. Section \ref{sec:petrovsky} is dedicated to the Petrovsky criterion, that is the proof of Theorem \ref{thm:petrovsky} and Theorem \ref{thm:not regular}.
 \subsection{Notation}
 In what follows $\Omega$ is an \emph{arbitrary} domain in $\mathbb{R}^n \times (-\infty, \infty)$. $Q_T$ is a space-time cylinder: $Q_T = Q\times (0, T)$, $\partial \Omega$ is the Euclidean boundary of $\Omega$ and $\partial_pQ_T$ is the \emph{parabolic} boundary of $Q_T$, i.e. $(\overline{Q}\times \{0\}) \cup (\partial Q\times (0, T])$. (the "bottom" and the sides of the cylinder. The top is excluded.) $\zeta, \eta \in \mathbb{R}^n \times \mathbb{R}$ are points in space-time, that is $\zeta =(x, t)$.
 
 We denote by $Du$ the gradient of $u(x, t)$  taken with respect to the spatial coordinates $x$, and $D^2u$ is the spatial Hessian matrix of $u$. $\langle a, b \rangle$ is the Euclidean inner product of the vectors $a, b \in \mathbb{R}^n$.
 and $x \otimes y$ denotes the tensor product of the vectors $x, y$, that is $(x\otimes y)_{i, j}=x_iy_j$.
 
   The space of lower semi-continuous functions from $\Omega$ to $\mathbb{R}\cup \{\infty\}$ is denoted by $\operatorname{LSC}(\Omega)$, while $\operatorname{USC}(\Omega)$ contains the upper semicontinuous ones.
\section{Several solutions}
\label{sec:severalsolutions}
\noindent We derive several explicit solutions to \eqref{eq:evoplap}, and identify the fundamental solution.
\subsection{Uniform propagation}
Assume $u(x, t) = w(\langle a, x \rangle-bt)$, $a\in \mathbb{R}^n$, $b\in \mathbb{R}$. We then get 
\begin{align*}
& u_t =-bw' ,\\
& u_{x_i} =a_iw', \\
& u_{x_ix_j}=a_ia_jw'',
\end{align*}
and hence 
\[
\Delta u =\Delta_{\infty}^N =|a|^2w''.
\]
Inserting this into \eqref{eq:evoplap}, we get that $w$ must satisfy
\[
w'+w''|a|^2\frac{p-1}{bp} =0,
\]
with solution
\[
w(\zeta) = A+B\operatorname{e}^{-\frac{\zeta}{m}},
\]
or
\[
u(x, t) = A+B\operatorname{e}^{-\frac{1}{m}(\langle a, x \rangle-bt)},
\]
with $m =|a|^2\frac{p-1}{bp}$.
 and $\zeta =\langle a ,x \rangle-bt$.
\subsection{Separable solution}
Assume $u(x, t)=f(r)+g(t)$, $r=|x|$. We get
\begin{align*}
u_t =g', \\
\Delta u = f''+\frac{n-1}{r}f', \\
\Delta_{\infty}^Nu = f''.
\end{align*}
Thus
\[
g'(t) =c,
\]
and
\[
f''(r)+\frac{n-1}{(p-1)r}f'(r) =c
\]
So
\[
f(r)=  \frac{1}{2}\frac{cp}{n+p}r^2 + c_1\frac{p-1}{p-n}r^{\frac{p-n}{p-1}}+c_2,
\]
Setting $c_1=c_2=0$, we get the solution
\[
u(x, t) =c\left(\frac{p}{n+p}|x|^2+2\frac{p-1}{p-n}|x|^{\frac{p-n}{p-1}}+2t\right).
\]
%\subsection{Separable solution II}
%Now the ansatz is $u(x, t)=f(r)g(t)$. We calculate
%\begin{align*}
%& u_t=g'(t)f(r), \\
%& Du = f'(r)g(t)\frac{x}{|x|}, \\
%& D^2u = g(t)\left(f''(r)\frac{x \otimes x}{|x|^2}+\frac{f'(r)}{|x|^2}I- f'(r)\frac{x \otimes x}{|x|^3}\right) .
%\end{align*}
%Hence
%\[
%\text{tr}(D^2u) = g(t)\left(f''(r)+\frac{nf'(r)}{r^2}-\frac{f'(r)}{r}\right)
%\]
%and
%\[
%\Delta_{\infty}^Nu = g(t)f''(r).
%\]
%Inserting all this into \eqref{eq:evoplap} yields
%\[
%f(r)g'(t) =g(t)\left(\beta f''(r)+\left(\frac{n}{r}-1\right)f'(r)\right),
%\]
%and upon separation:
%\[
%g'(t)=kg(t)
%\]
%and 
%\[
%f(r)=k\left(\beta f''(r)+\left(\frac{n}{r}-1\right)f'(r)\right).
%\]
\subsection{Heat Equation transformation}
We search for solutions on the form
$u(x, t) =v(r^{\nu}, t)$, where $r=|x|$ and $\nu$ is a critical exponent to be determined. This gives
\begin{align*}
&\Delta u = \frac{\partial^2}{\partial r^2}v(r^{\nu}, t) + \frac{n-1}{r}\frac{\partial}{\partial r}v(r^{\nu}, t), \\
&\Delta_{\infty}^Nu =  \frac{\partial^2}{\partial r^2}v(r^{\nu}, t).
\end{align*}
Calculating, we get
\begin{align*}
&\frac{\partial}{\partial r}v(r^{\nu}, t) = \nu r^{\nu-1}v', \\
&\frac{\partial^2}{\partial r^2}v(r^{\nu}, t) = \nu(\nu-1)r^{\nu-2}v'+\nu^2r^{2\nu-2}v'',
\end{align*}
where 
\[
v' =\frac{\partial v}{\partial \rho}, \ \rho =r^{\nu}
\]
Hence we get
\begin{align*}
&\Delta u = \nu(\nu-1)r^{\nu-2}v'+\nu^2r^{2\nu-2}v''+\nu r^{\nu-2}v', \\
&\Delta_{\infty}^Nu = \nu(\nu-1)r^{\nu-2}v'+\nu^2r^{2\nu-2}v''.
\end{align*}
Inserting this into \eqref{eq:evoplap} and collecting terms we get
\begin{equation}
\label{eq:almostheat}
v_t= \nu^2\frac{p-1}{p}r^{2\nu-2}v''+\frac{\nu}{p}r^{\nu-2}[(p-2)(\nu-1)+n+\nu]v'.
\end{equation}
We want to eliminate the first order terms in \eqref{eq:almostheat}, so we demand
\[
(p-2)(\nu-1)+n+\nu = 0,
\]
or 
\[
\nu = \frac{p-n}{p-1}.
\]
Then \eqref{eq:evoplap}  reads
\[
v_t =\frac{(p-n)^2}{p(p-1)}\rho^{\frac{2-n}{p-n}}v_{\rho \rho},
\]
where 
\[
\rho = |x|^{\frac{p-n}{p-1}}.
\]
\subsection{Similarity I}
We make the ansatz
\[
u(x, t) =F(\zeta), \ \ \zeta =\frac{|x|^2}{t},
\]
We calculate
\begin{align*}
& u_t =-\frac{F'(\zeta)\zeta}{t}, \\
& Du =\frac{2F'(\zeta)x}{t}, \\
& D^2u =\frac{2F'(\zeta)}{t}I+\frac{4F''(\zeta)}{t^2}(x \otimes x).
\end{align*}
Hence
\begin{align*}
\text{tr}(D^2u) =\frac{2F'(\zeta)}{t}n+\frac{4F''(\zeta)}{t^2}|x|^2, \\
\Delta_{\infty}^Nu =\frac{2F'(\zeta)}{t}+\frac{4F''(\zeta)}{t^2}|x|^2.
\end{align*}
So, if $u$ is a solution then
\[
-\frac{F'(\zeta)\zeta}{t}-\alpha \frac{F'(\zeta)}{t}- \beta\frac{F''(\zeta)\zeta}{t} =0,
\]
with $\alpha, \beta$ as in \eqref{eq:bestsolution}. If $t \neq 0$;
\[
\frac{F''(\zeta)}{F'(\zeta)}=\frac{d}{d\zeta}\log F'(\zeta) = -\frac{\alpha}{\beta \zeta} -\frac{1}{\beta}.
\]
Integrating the above gives
\[
\log F'(\zeta) =-\frac{\zeta}{\beta}-\frac{\alpha}{\beta}\log\zeta,
\]
or
\[
F(\zeta) = C\int_0^{\zeta}s^{-\frac{\alpha}{\beta}}\operatorname{e}^{-\frac{s}{\beta}} \ ds.
\]
This leads to the solution
\begin{equation}
\label{eq:tardsoln}
u(x, t) = C\int_{0}^{|x|^2/t}s^{-\frac{\alpha}{\beta}}\operatorname{e}^{-\frac{s}{\beta}} \ ds.
\end{equation}
This solution is not differentiable where $x=0$. However, \eqref{eq:tardsoln} is a solution outside the line $\{0\}\times(0, \infty)\subset \mathbb{R}^n\times (0, \infty)$, and a subsolution or supersolution depending on the sign of $C$ in all of $\mathbb{R}^n\times (0, \infty)$.
\subsection{Similarity II}
\label{sec:similarity1}
We note that if $u(x, t)$ is a solution of \eqref{eq:evoplap}, then so is $v(x, t) =u(Ax, A^2t)$. We search for solutions on the form
\[
u(x, t) =g(t)f(\zeta), \ \ \zeta =\frac{|x|^2}{t}.
\] 
Inserting this into \eqref{eq:evoplap}, we get
\[
u_t =g'(t)f(\zeta) -\frac{g(t)f'(\zeta)\zeta}{t},
\]
\[
Du = \frac{2g(t)f'(\zeta)}{t}x,
\]
and
\[
D^2u = \frac{2g(t)f'(\zeta)}{t} I +\frac{4g(t)f''(\zeta)}{t^2}(x\otimes x).
\]
Hence we see
\[
\text{tr}(D^2u) = \frac{2g(t)f'(\zeta)}{t} n+\frac{4g(t)f''(\zeta)}{t^2}|x|^2,
\]
and

\[
\Delta_{\infty}^Nu =  \frac{2g(t)f'(\zeta)}{t} +\frac{4g(t)f''(\zeta)}{t^2}|x|^2.
\]
Therefore, $u$ is a solution to \eqref{eq:evoplap} if
\[
g'(t)f(\zeta)-\frac{g(t)f'(\zeta)\zeta}{t} =2\frac{p+n-2}{p}\frac{g(t)f'(\zeta)}{t}+4\frac{p-1}{p}\frac{g(t)f''(\zeta)\zeta}{t}.
\]
 with $\alpha = 2\ \frac{p+n-2}{p}$ and $\beta =4\ \frac{p-1}{p}$,
\[
tg'(t)f(\zeta)-\alpha g(t)f'(\zeta) =g(t)\zeta(f'(\zeta)+\beta f''(\zeta)).
\]
for $t>0$. The right hand side of this is zero if $f(\zeta)=\operatorname{e}^{-\frac{\zeta}{\beta}}$. Inserting this back in, we see that
\[
f(\zeta) \left(tg'(t)+\frac{\alpha}{\beta}g(t)\right) =0, 
\]
with solution
\[
g(t) =t^{-\frac{\alpha}{\beta}}.
\]
Together, this gives
\begin{equation}
\label{eq:bestsolution}
\begin{split}
& u(x, t) =t^{-\frac{\alpha}{\beta}}\operatorname{e}^{-\frac{|x|^2}{\beta t}}, \\
 &\alpha = 2\ \frac{p+n-2}{p}, \ \beta =4\ \frac{p-1}{p}.
\end{split}
\end{equation}
This is a solution for $t>0$, and if we replace $t$ by $-t$ we get a solution for negative $t$, as well.\footnote{It has recently come to the author's attention that this solution also was found in \cite{banerjee2013gradient}.}
\begin{remark}
As $p\to \infty$, $ \alpha \to 2 $ and $\beta \to 4$. This gives that \eqref{eq:bestsolution} converges to the fundamental solution 
\[
W(x, t) =\frac{1}{\sqrt{t}}\operatorname{e}^{-\frac{|x|^2}{4t}}.
\]
of the normalised $\infty$-parabolic equation found in \cite{Juutinen2006}. Compare also \eqref{eq:bestsolution} with the fundamental solution to the heat equation,
\[
H(x, t) = \frac{1}{(4\pi t)^{n/2}}\mathrm{e}^{-\frac{|x|^2}{4t}}.
\]
\end{remark}
\section{Comparison and Perron solutions}
\label{sec:comparison}
\noindent In this Section we present several basic facts regarding the existence of solutions to \eqref{eq:evoplap}, and present Perron's method. We start with the definition of viscosity solutions. If $Du=0$, we replace the operator $\mathcal{A}_p$ with its lower or upper  semicontinuous envelope:
	\begin{definition}
		\label{def:viscsuper}
		A lower semicontinuous function $u \in L^{\infty}(\Omega)$ is  a \emph{viscosity supersolution} of \eqref{eq:evoplap} provided that, if $ u - \phi$ has a minimum at $\zeta_0 \in \Omega$ for  $\phi \in C^2 (\Omega)$, then 
		\[
		\begin{cases}
		\phi_t(\zeta_0)-\mathcal{A}_p\phi(\zeta_0) \geq 0, \ &\text{if} \ D\phi(\zeta_0)\neq 0, \\
		\phi_t(\zeta_0) -\frac{1}{p}\text{tr}(D^2\phi(\zeta_0))-\frac{p-2}{p}\lambda(D^2\phi(\zeta_0))\geq 0, \ &\text{if} \ D\phi(\zeta_0)=0. 
		\end{cases}		
		\]
			An upper semicontinuous function $u \in L^{\infty}(\Omega)$ is  a \emph{viscosity subsolution} of \eqref{eq:evoplap} provided that, if $ u - \phi$ has a maximum at $\zeta_0 \in \Omega$ for  $\phi \in C^2 (\Omega)$, then 
		\[
		\begin{cases}
		\phi_t(\zeta_0)-\mathcal{A}_p\phi(\zeta_0) \leq 0, \ &\text{if} \ D\phi(\zeta_0)\neq 0, \\
		\phi_t(\zeta_0) -\frac{1}{p}\text{tr}(D^2\phi(\zeta_0))-\frac{p-2}{p}\Lambda(D^2\phi(\zeta_0))\leq 0, \ &\text{if} \ D\phi(\zeta_0)=0. 
		\end{cases}		
		\]
		
		A function that is both a viscosity sub- and supersolution is called a \emph{viscosity solution}.
		
		Here $\lambda(D^2\phi(\zeta_0))$, $\lambda(D^2\phi(\zeta_0))$ denotes the smallest and largest eigenvalues of the Hessian matrix $D^2\phi(\zeta_0)$, and $\text{tr}(D^2\phi(\zeta_0))$ is its trace. 
	\end{definition}
It turns out that the second condition in Definition \ref{def:viscsuper} can be relaxed. This  is the Lemma 2 in \cite{manfredi2010asymptotic}.
\begin{lemma}
	\label{def:p-parabolic1}
	An upper semicontinuous function $u \in L^{\infty}(\Omega)$, is a viscosity subsolution of \eqref{eq:evoplap} provided that, if $ u - \phi$ has a maximum at $\zeta_0 \in \Omega$ for  $\phi \in C^2 (\Omega)$, then either
	\[
	\phi_t(\zeta_0)-\mathcal{A}_p\phi(\zeta_0) \leq 0, \ \text{if} \ D\phi(\zeta_0)\neq 0,
	\]
	or 
	\[
	\phi_t(\zeta_0) \leq 0, \ \text{if} \ D\phi(\zeta_0)=0, \ D^2\phi(\zeta_0) =0.
	\]
	A similar result holds for viscosity supersolutions.
\end{lemma}
 We define $p$-parabolic functions in $\Omega$ as follows:
\begin{definition}
\label{def:p-parabolic2}
A function $u\in LSC(\Omega)\cap L^{\infty}(\Omega)$ is  a \emph{supersolution} to \eqref{eq:evoplap} if
it satisfies the following comparison principle: \\

 On each set of the form $Q_{t_1, t_2} =Q \times (t_1, t_2 )$  with closure in $\Omega$, and for each solution $h$ to \eqref{eq:evoplap} continuous up to the closure of $Q_{t_1, t_2}$, and $h \leq u $ on $\partial_p Q_{t_1, t_2}$, then  
 \[
 h \leq u \ \text{in}  \ Q_{t_1, t_2}.
 \]
\end{definition}
In Banerjee–Garofalo \cite{banerjee2015dirichlet}, they use the name \emph{generalized super/subsolution} instead of super/subparabolic function. They prove that these are the same as the viscosity super/subsolutions in a given domain. Hence we can use the term parabolic interchangeably with viscosity solution.

The assumption that supersolutions are bounded is not needed. See Theorem 2.6 in \cite{bjorn2017tusk}.

We shall improve this to include comparison on general domains $\Sigma$ compactly contained in $\Omega$, see Lemma \ref{thm:bettercomparison}.

Using the classical comparison principle for cylindrical domains $Q_T=Q\times (0, T)$ and a covering argument, we can prove Theorem 3.10 from \cite{banerjee2015dirichlet}. This is the comparison principle, essential for Perron's method to work.
\begin{theorem}
\label{thm:comparison}
Suppose $u$ is a supersolution bounded from  above and $v$ is a subsolution bounded from below of \eqref{eq:evoplap} in a bounded
open set $\Omega \subset \mathbb{R}^{n+1}$. If at each point $\zeta_0 \in \partial \Omega$ we have
\[
\limsup_{\zeta \to \zeta_0} v(\zeta) \leq \liminf_{\zeta \to \zeta_0} u(\zeta),
\]
then $v \leq u$ in $\Omega$.
\end{theorem}
%We also need the existence of such solutions to the Dirichlet problem. This is derived from Theorem 2.6 in \cite{banerjee2015dirichlet}.
%\begin{theorem}[Existence]
%\label{thm:existence}
%Let $Q \subset \mathbb{R}^n$ be a bounded domain such that it satisfies an (uniform)
%exterior cone condition at each point on the boundary. Furthermore, we assume that
%there exists an exhaustion of $Q$ by compact subdomains $Q_k$ such that each point
%of $\partial Q_k$ can be touched by an exterior cone of fixed size independent of $k$. Let $g$ be a
%continuous function on $\partial_pQ_T$ . Then, there exists a unique solution to \eqref{eq:evoplap} in $\overline{Q}_T$
%such that $u = g$ on $\partial_pQ_T$.
%\end{theorem}

\subsection{The Perron Method}
\label{subsec:perron}
We start  with a definition. Let $f : \partial \Omega \to \mathbb{R}$ be a continuous function.
\begin{definition}
A function $u$ belongs to the \emph{upper class} $\mathcal{U}_f$ if $u$ is a viscosity supersolution in $\Omega$ and
\begin{equation*}
\liminf_{\eta \to \zeta}u(\eta) \geq f(\zeta)
\end{equation*}
for $\zeta \in \partial \Omega$.

Likewise, a function $v$ belongs to the \emph{lower class} $\mathcal{L}_f$ if $v$ is a viscosity subsolution in $\Omega$ and
\begin{equation*}
\limsup_{\eta \to \zeta}v(\eta) \leq f(\zeta)
\end{equation*}
for all $\zeta \in \partial \Omega$.

We define the \emph{upper solution} 
\[
\overline{H}_f(\zeta)=\inf\{u(\zeta) : u \in \mathcal{U}_f \},
\] 
and the \emph{lower solution} 
\[
\underline{H}_f(\zeta)=\sup\{v(\zeta) : v \in \mathcal{L}_f \}.
\]
 Note that at each point the $\inf$ and $\sup$ are taken over the \emph{functions}.
 \end{definition}
\begin{remark}
The Comparison Principle, Theorem \ref{thm:comparison}, gives immediately that $v \leq u$ in $\Omega$, for $v \in\mathcal{L}_f$ and $u \in \mathcal{U}_f$, and hence 
\[\underline{H}_f\leq \overline{H}_f.\]
 Whether $\underline{H}_f= \overline{H}_f$ holds in general, is a more subtle question.
\end{remark}
We need that the Upper and Lower Perron solutions, indeed, are viscosity solutions to \eqref{eq:evoplap}. 
\begin{theorem*}[Theorem 3.12 in \cite{banerjee2015dirichlet}]
 The upper Perron solution $\overline{H}_f$ and the lower Perron solution $\underline{H}_f$
are solutions to \eqref{eq:evoplap} in $\Omega$.
\end{theorem*}
An integral part of the theory of Perron solutions are the boundary regularity and the barrier functions.
\begin{definition}
We say that $\zeta_0 \in \partial \Omega $ is a \emph{regular} boundary point if 
\[
\lim_{\zeta \to \zeta_0}\underline{H}_f(\zeta) =f(\zeta_0).
\]
for every continuous function $f: \partial \Omega \to \mathbb{R}$.
\end{definition}
Note that we instead could have used $\overline{H}_f$ in the above, since $\underline{H}_f =-\overline{H}_{-f}$.
\begin{remark}
The Petrovsky condition in Theorem \ref{thm:not regular} shows that a point can be regular for $u_t=\mathcal{A}_pu$ but \emph{not} for $u_t=\mathcal{A}_qu$, $p<q$. Hence it would be more accurate to use the term $p$-\emph{regular}, but we  use regular where no confusion will arise. 
\end{remark}
\begin{definition}
	\label{def:barrier}
A function $w$ is a barrier at $\zeta_0 \in \partial \Omega$ if
\begin{enumerate}
\item $w >0$ and $w$ is $p$-superparabolic in $\Omega$,
\item $\liminf_{\zeta \to \eta}w(\zeta) >0$ for $\zeta_0 \neq \eta \in \partial \Omega$,
\item $\lim_{\zeta \to \zeta_0}w(\zeta) =0.$
\end{enumerate}
\end{definition}
Using barrier functions, we can prove the following classical result, which is 
Theorem 4.2 in \cite{banerjee2015dirichlet}:
\begin{theorem}
\label{thm:barrier}
A boundary point $\zeta_0$ is regular if and only if there exists a barrier at $\zeta_0.$
\end{theorem}
The existence of a barrier is a local property in the following sense: Let $\tilde{\Omega}$ be another domain such that
\[
\overline{B}\cap \tilde{\Omega} =\overline{B}\cap \Omega
\]
for an open ball $B$ centered at $\zeta_0$. Suppose there is a barrier $w$ at $\zeta_0$, and let
\[
m= \inf\{w(\zeta)\ : \ \zeta \in \partial B \cup \tilde{\Omega}\}.
\]
It now follows that the function
\[
v=
\begin{cases}
\min(w, m) \ &\text{in} \ B\cup \tilde{\Omega}, \\
m \ &\text{in} \ \Omega \setminus B
\end{cases}
\]
is a barrier in $\Omega$.  Indeed, since $w$ is assumed to be a barrier in $\tilde{\Omega}$, $\left.w \right|_{\tilde{\Omega}}> 0$ by the definition, and therefore $m>0$ and $v>0$. Since $\tilde{\Omega}\cap \overline{B}=\Omega \cap \overline{B}$ we see that $\liminf_{\eta \to \zeta}v(\eta)=\liminf_{\eta \to \zeta}w(\eta)>0$ on $\partial \tilde{\Omega} \cap \overline{B}$ and $\liminf_{\eta \to \zeta}v(\eta)=m>0$ elsewhere. At last, $\lim_{\zeta \to \zeta_0}v(\zeta)=\lim_{\zeta \to \zeta_0}w(\zeta)=0$.  
From this we get the following useful corollary.
\begin{corollary}
	\label{cor:irregular}
	Let $\tilde{\Omega} \subset \Omega$, and let $\zeta_0$ be a common boundary point. If $\zeta_0$ is \emph{not} a regular point for $\tilde{\Omega}$, then it is not a regular point for $\Omega$.
\end{corollary}
\begin{proof}
	Let $\tilde{\Omega} \subset \Omega$, and let $\zeta_0$ be an irregular boundary point for $\tilde{\Omega}$. Assume that $\zeta_0$ is regular for $\Omega$. Then Theorem \ref{thm:barrier} gives that there exists a barrier, $w$, in $\Omega$. The above implies the existence of a barrier in $\tilde\Omega$, contradicting the irregularity of $\zeta_0$.
\end{proof}
%Using classical regularity theory for parabolic equations, we get the following convergence result.
A classical application of the theory of viscosity solutions is the following convergence lemma. 
\begin{lemma}
\label{thm:convergenceofeq}
Assume that $\{u_p\}_p$ is a sequence of viscosity solutions of 
\[
u_t-\mathcal{A}_pu=0.
\]
Assume further that $\{u_p\}_p$ contains a subsequence $\{u_{p_j}\}_j$ that converges uniformly to a function $u_{\infty}$ in $\Omega$. Then, as $j \to \infty$, the  $u_{p_j}$
converge to $u$, the viscosity solution of the normalised $\infty$-parabolic equation \eqref{eq:norminftylaplace}, that is
\[
u_t-\Delta^N_{\infty}u=0.
\]
\end{lemma}
\begin{proof}
We show that viscosity subsolutions of \eqref{eq:evoplap} converge to viscosity subsolutions of \eqref{eq:inftylaplace}. The proof for supersolutions is similar.
%Theorem 4.1 and Theorem 4.2 in \cite{does2009evolution} show that there exists a constant $C$ so that
%\[
%|u_p(\zeta)-u_p(\eta)| \leq C|\zeta-\eta|,
%\]
%for $\zeta, \eta$ in a cylinder $Q_T$. 
%We show first that the sequence $u_p$ is equicontinuous. Arguing as in Theorem 2.9 in \cite{banerjee2015dirichlet}, we see that if $u_p$ is a solution to \eqref{eq:evoplap}, then at a maximum point for $u_p-\phi$ we get
%\[
%\phi_t \leq a_{ij}(D\phi)D_{ij}\phi
%\]
%where the coefficient matrix $[a_{ij}]$ satisfies
%\[
%\lambda_p I =\min\left(\frac{1}{p}, \frac{p-1}{p}\right)I \leq [a_{ij}] \leq \max\left(\frac{1}{p}, \frac{p-1}{p}\right)I= \Lambda_p I.
%\]
%Therefore, by the results in \cite{wang1992regularity}, we get the following Hölder bounds for every compact set $K\subset \Omega$
%\[
%|u_p(x, t)-u_p(y, s)| \leq C\|u_p\|_{\infty, K}(|x-y|+|t-s|^{\frac{1}{2}})^{\alpha}.
%\]
%where $\alpha$ depends upon $K, n$ and $\lambda_p, \Lambda_p$. Since $\lambda_p, \Lambda_p \to 0, 1$ as $p \to \infty$, we have that $\alpha \in (0, 1)$ for every $p$.
%
%Since $u_p$ is assumed to be bounded, we get that the sequence $u_p$ is equicontinuous, and hence the Arzelà–Ascoli theorem gives the existence of a subsequence $u_{p_j}$ that converges locally uniformly to a continuous function $u_{\infty}.$ This $u_{\infty}$ is our candidate for a viscosity subsolution of \eqref{eq:norminftylaplace}.
We say that $u \in \operatorname{USC}(\Omega)$ is a \emph{viscosity subsolution} to \eqref{eq:norminftylaplace} if, for every function $\phi \in C^2(\Omega)$ such that $u-\psi$ has a maximum at $\zeta_0$, we have
\begin{equation*}
\begin{cases}
\phi_t(\zeta_{\infty})-\Delta_{\infty}^N \psi(\zeta_{\infty})  \leq 0, \ \text{for} \ D\phi(\zeta_{\infty}) \neq 0 \\
\phi_t(\zeta_{\infty}) -\Lambda(D^2\psi(\zeta_{\infty})) \leq 0, \ \text{for} \ D\phi(\zeta_{\infty}) =0.
\end{cases}
\end{equation*}
Assume that $u_{\infty}-\phi$ has a maximum at $\zeta_{\infty}$ for $\phi \in C^2(Q_T)$.

{\bf 1.} Assume first that $D\phi(\zeta_{p_j})\neq0$ for $j$ greater than some number $N$. By definition of viscosity subsolution, we then have

\begin{equation}
\label{eq:converges}
\phi_t(\zeta_{p_j})- \frac{1}{{p}_j}\Delta \phi(\zeta_{p_j}) -\frac{p_j-2}{p_j}\Delta^N_{\infty}\phi(\zeta_{p_j}) \leq 0. 
\end{equation}
Since $u_{p_j}\to u_{\infty}$ uniformly, standard arguments, cf \cite{lindqvist2016notes} gives that the maximum points $\zeta_{p_j}$ converge to a maximum point $\zeta_{\infty}$ of $u_{\infty}-\phi$. Hence, letting $j\to \infty$ in \eqref{eq:converges} we see that
\[
\phi_t(\zeta_{\infty})-\Delta_{\infty}^N\phi(\zeta_{\infty}) \leq 0.
\]

{\bf 2. } If $D\phi(\zeta_{p_j})=0$ for $j>N$, we have
\[
	\phi_t(\zeta_{p_j}) -\frac{1}{p}\text{tr}(D^2\phi(\zeta_{p_j}))-\frac{p-2}{p}\Lambda(D^2\phi(\zeta_{p_j}))\leq 0
\]
and arguing as in the first case, we get as $p_j \to \infty$;
\[
	\phi_t(\zeta_{\infty})-\Lambda(D^2\phi(\zeta_{\infty}))\leq 0.
\]
This shows that $u_{\infty}$ is indeed a viscosity subsolution of the normalised $\infty$-parabolic equation.
\end{proof}
\begin{remark}
The existence of a uniformly convergent subsequence of $u_p$ is not known to exist in general. Does finds such an example for the  initial-boundary value problem with smooth boundary data in \cite{does2009evolution}.
\end{remark}
\section{Exterior Sphere Condition}
\label{sec:exterior sphere}
\noindent We use the barrier characterization to prove Theorem \ref{thm:exterior sphere}. We repeat the result here for completeness.
\begin{theorem*}[Exterior sphere]
Let $\zeta_0 =(t_0, x_0) \in \partial \Omega$, and suppose that there exists a closed ball $\{(x, t) \ : \ |x-x'|^2+(t-t')^2 \leq R_0^2\}$ intersecting $\overline{\Omega}$ precisely at $\zeta_0$. Then $\zeta_0$ is regular, if the intersection point is \underline{not} the south pole, that is $(x_0, t_0)  \neq (x', t'-R_0)$.
\end{theorem*}

\begin{proof}
We use the exterior sphere to construct a suitable barrier function at $\zeta_0$. Define
\[
w(x, t) = \operatorname{e}^{-aR^2_0}-\operatorname{e}^{-aR^2}, \ R^2= |x-x'|^2+(t-t')^2,
\]
for a constant $a>0$ to be determined. Clearly $w(x_0, t_0)=0$, and close to $(x_0, t_0)$ we have 
\begin{equation}
\label{eq:ballestimates}
\delta <|x-x'|, \ \ -2R_0 <t-t'.
\end{equation}

We prove that $w$ is a viscosity supersolution. Calculating the derivatives, we get
\begin{equation}
\begin{split}
\label{eq:w-derivatives}
&Dw(x, t) = 2a\mathrm{e}^{-aR^2}(x-x'), 
\\ &w_t(x, t)=2a\mathrm{e}^{-aR^2}(t-t'), \\
&D^2w(x, t)=2a\mathrm{e}^{-aR^2}(\mathbb{Id}_n-2a(x-x')\otimes(x-x')). 
\end{split}
\end{equation}
This shows that $Dw=0$ precisely when $x=x'$. According to Definition \ref{def:viscsuper}, we need to check the cases $x=x'$ and $x\neq x'$ separately.

{\bf 1.} Assume that $x\neq x'$. Then the point of contact is not the north pole. It suffices to show that $w$ is a classical supersolution.
 Inserting the derivatives into \eqref{eq:evoplap} we get
\begin{align*}
 & w_t-\frac{1}{p}\Delta w +\frac{p-2}{p}\Delta_{\infty}^Nw  \\
&=2a(t-t')\operatorname{e}^{-aR^2}-\frac{1}{p}\operatorname{e}^{-aR^2}\left[(p-1)(2a-4a^2r^2)+2ar\frac{n-1}{r}\right] \\
&= 2a\operatorname{e}^{-aR^2}\left[(t-t')+2a\frac{p-1}{p}|x-x'|^2-\frac{p+n-2}{p}\right].
\end{align*}
In light of \eqref{eq:ballestimates}, we have 
\begin{equation*}
\label{eq:mustbepos}
w_t-\mathcal{A}_pw> 2a\operatorname{e}^{-aR^2}\left[-2R_0+2a\frac{p-1}{p}\delta^2-\frac{p+n-2}{p}\right]
\end{equation*}
For the right hand side of this to be positive, we must have
\[
-R_0+a\frac{p-1}{p}\delta^2>\frac{p+n-2}{2p},
\]
and choosing $a$ big enough to ensure this, shows that $w$ is superparabolic.

{\bf 2.} If the point of intersection is the north pole, i.e $(x_0, t_0) = (x', t'+R_0)$, we can find points arbitrarily close to the line $x=x'$ such that 
\[
w_t-\mathcal{A}_pw =2a\operatorname{e}^{-aR^2}\left[R_0-\frac{p+n-2}{p}\right] +\epsilon,
\]
for any $\epsilon >0$.
We see that we must demand that the radius $R_0$ satisfies
\[
R_0 \geq \alpha/2, \ \alpha = 2\frac{p+n-2}{p}
\]
for $w$ to be a barrier in this case. 

Assume now that $x=x'$. We need to verify that for every $\phi\in C^2(\Omega)$ touching $w$ from below at $(x', t)$ we have
\begin{equation}
\label{eq:w-visc-super}
\phi_t(x', t) \geq\frac{1}{p}\text{tr}(D^2\phi(x', t))+ \frac{p-2}{p} \lambda(D^2\phi(x', t)).
\end{equation}
Assume to the contrary that there is a $\phi$ such that $w-\phi$ has a minimum at $(x', t)$, but that
\[
\phi_t(x', t) <\frac{1}{p}\text{tr}(D^2\phi(x', t))+ \frac{p-2}{p} \lambda(D^2\phi(x', t)).
\]
Since $w-\phi$ has a minimum, we must have
\[
\phi_t(x', t)=u_t(x', t), \ D\phi(x', t)=Du(x', t), \ D^2u(x', t) \geq D^2\phi(x', t).
\]
This implies, for any $z\in \mathbb{R}^n$
\[
\langle D^2w\ z, z\rangle \geq \langle D^2\phi \ z, z\rangle
\]
and, since $D^2w$ is a scalar multiple of the identity matrix, 
\[
\text{tr}(D^2w)|z|^2\geq \text{tr}(D^2\phi)|z|^2
\]
at $(x', t)$. Hence
\begin{align*}
\frac{1}{p}\text{tr}(D^2w)|z|^2 +\frac{p-2}{p}\langle D^2w\ z, z\rangle \\
\geq \frac{1}{p}\text{tr}(D^2\phi)|z|^2 +\frac{p-2}{p}\langle D^2\phi\ z, z\rangle \\
\geq \frac{1}{p}\text{tr}(D^2\phi)|z|^2 +\frac{p-2}{p}\lambda(D^2\phi)|z|^2 \\
>\phi_t|z|^2=w_t|z|^2.
\end{align*}
Inserting $x=x'$ in \eqref{eq:w-derivatives} and dividing by $|z|^2$ this is
\[
\frac{1}{p}2an\operatorname{e}^{-aR^2}+\frac{p-2}{p}2a\operatorname{e}^{-aR^2} > 2a\operatorname{e}^{-aR^2}(t-t'),
\]
or
\[
\frac{n+p-2}{p}>(t-t').
\]
This is a contradiction because of our restriction on the radius, and hence \eqref{eq:w-visc-super} must hold, and $w$ is a supersolution even in this case.

The condition $R_0 \geq \alpha/2$ restricts the set of exterior spheres usable in a positive way. The author does not know if this restriction can be circumvented.

The exclusion of the south pole $(x_0, t_0)=(x', t'-R_0)$ in the above is strictly necessary, since then for $(x, t)$ close to $(x_0, t_0)$ we could have
\[
(t-t')<0 \ \text{and} \ |x-x'|=|x'-x'|=0,
\]
and so 
\begin{align*}
w_t-\mathcal{A}_pw =2a\operatorname{e}^{-aR^2}\left[(t-t')-\frac{p+n-2}{p}\right] <0
\end{align*}
for any positive $a$, since $p+n\geq 2$.
\end{proof}
Another way to see that it is necessary to exclude the south pole is to consider the Dirichlet problem on the cylinder $Q_T =Q\times (0, T)$.
\begin{example}
	Suppose that $f : \partial Q_T\to \mathbb{R}$ is continuous. Theorem \ref{thm:comparison} and Theorem 2.6 in \cite{banerjee2015dirichlet} gives the existence of a unique viscosity solution $h$ in $Q_T$.
	
	Now construct the upper and lower Perron solutions $\overline{H}_f$ and $\underline{H}_f$. Since both are $p$-parabolic in $Q_T$, uniqueness gives that $\underline{H}_f = \overline{H}_f=h$, regardless of what values we choose at that part of the boundary where $t=T$. Indeed, $h$ itself need not be in either the upper or lower class, because we may not have that either $h>f$ or $h<f$ on the plane $t=T$. However, if we define
	\begin{equation*}
	\tilde{h} =h(x, t)+\frac{\epsilon}{T-t},
	\end{equation*}
	we see that
	\begin{equation*}
	\tilde{h}_t-\mathcal{A}_p\tilde{h} =0+\frac{\epsilon}{(T-t)^2},
	\end{equation*}
	so $\tilde{h}$ is in $\mathcal{U}_f$ for $\epsilon >0$, and in $\mathcal{L}_f$ for $\epsilon <0$. Therefore, it is possible for every point on the top of the cylinder to be irregular. we can say that $f$ is \emph{resolutive} in this case.
\end{example} 
We provide another example of an irregular boundary point.
\begin{example}[Latest moment on heat balls]
	Recall the self-similar solution derived in Section \ref{sec:similarity1}. We define the \emph{fundamental} solution to \eqref{eq:evoplap} as
	\begin{equation*}
	H_p(x, t) = t^{-\frac{\alpha}{\beta}}\operatorname{e}^{-\frac{|x|^2}{\beta t}}.
	\end{equation*}
	Analogous to the heat equation and the $p$-parabolic equation, we define the \emph{normalised} $p$-\emph{parabolic} \emph{balls} by the  level sets
	\begin{equation}
	\label{eq:level sets}
	H_p(x_0-x, t_0-t) >c
	\end{equation}
	We want to prove that the latest moment, or "centre" $(x_0, t_0)$ of \eqref{eq:level sets} is \emph{not} a  regular point. Fix $c>0$. We can assume that $(x_0, t_0)=(0,0)$, so that \eqref{eq:level sets} reads
	\[
	(-t)^{-\frac{\alpha}{\beta}}\operatorname{e}^{-\frac{|x|^2}{\beta (-t)}}>c
	\]
	for $t<0$. 
	But this is equivalent to
	\[
	|x|^2<t(\frac{\log{c}}{\beta}+\alpha\log|t|),
	\]
	and this inequality defines a domain containing the one in the Petrovsky criterion \ref{thm:petrovsky}. Hence the origin must be irregular.
\end{example}
We prove that it suffices to consider arbitrary domains in the equivalent definition of $p$-parabolic functions. The proof follows the same idea as in \cite{kilpelainen1996dirichlet}.
\begin{lemma}
\label{thm:bettercomparison}
A function $u\in LSC(\Omega)\cap L^{\infty}(\Omega)$ is $p$-superparabolic if and only if
for each domain $\Sigma$ with compact closure in $\Omega$, and for each solution $h \in C(\overline{\Sigma})$ to \eqref{eq:evoplap}, the condition  $h \leq u $ on $\partial \Sigma$ implies  $h \leq u$ in $\Sigma$.
\end{lemma}
\begin{proof}
Assume first that Definition \ref{def:p-parabolic2} holds. If $\Sigma$ is a box or finite union of boxes, the result is clearly true. The case where $\Sigma$ is arbitrary follows by covering the set $\{h \geq u+\epsilon\}$ with finitely many boxes.

For the other direction, let $Q_{t_1, t_2}$ be a box with closure in $\Omega$ and let $h \in C(\overline{Q}_{t_1, t_2})$ be $p$-parabolic, and so that $h \leq u$ on $\partial_pQ_{t_1, t_2}$. Assume that
\[
Q =(a_1, b_1) \times \cdots \times (a_n, b_n).
\]
Let $\delta>0$ be so that $\delta <t_2-t_1$, and choose a hyperplane $P_{\delta}$ such that the points $(x, t_2-\delta)$ with $x_1=a_1$ and $(y, t_2)$ with $y_1=b_1$ belong to $P_{\delta}$. let $\Sigma$ be the subset of  $Q_{t_1, t_2}$ that contains all the points below the hyperplane, that is all $(x, t)$ with $t<s$ and  $(x, s) \in P_{\delta}$.

The Exterior sphere condition Theorem \ref{thm:exterior sphere} immediately gives that every point on $\partial \Sigma$ is regular. Fix $\epsilon>0$, and choose $\delta$ so small that
\[
u(x, t) \geq h(x, t)-\frac{\epsilon}{t_2+\frac{\delta}{2}-t}
\]
for $(x, t) \in P_{\delta}\cap \Sigma$. Let $\overline{H}_{\theta}$ be the upper Perron solution in $\Sigma$ with
\[
\theta = h-\frac{\epsilon}{t_2+\frac{\delta}{2}-t}
\]
as boundary function.
Then $\overline{H}_{\theta}$ is continuous up to $\partial \Sigma$, and we have
\[
u\geq \overline{H}_{\theta}
\]
in all of $\Sigma$ since the inequality holds on $\partial \Sigma$. Hence
\[
u(x, t) \geq h(x, t)-\frac{\epsilon}{t_2+\frac{\delta}{2}-t}
\]
in $\Sigma$, and letting $\epsilon, \delta \to 0$, we get
\[
u \geq h
\]
in the box $Q_{t_1, t_2}$.
\end{proof}

\section{The Petrovsky Criterion}
\label{sec:petrovsky}
\noindent We provide the proof of the Petrovsky Criterion, repeated here for completeness.
\begin{theorem*}
	\label{thm:petrowski}
	The origin $(x, t)=(0, 0)$ is a regular point for \eqref{eq:evoplap}
	in the domain $\Omega$ enclosed by the hypersurfaces 
	\begin{equation}
	\label{eq:thedomain}
	\{(x, t)\in \mathbb{R}^n\times (-\infty, 0)\ : \ |x|^2 =-\beta t\log|\log|t||\} \ \text{and} \ \{t=-c\}, 
	\end{equation}
	for a small constant $0<c<1$. Recall that
	\[
	\beta =4\frac{p-1}{p}.
	\] 
\end{theorem*}
According to Theorem \ref{thm:barrier}, it suffices to find a barrier function $w$ so that
	\begin{enumerate}
		\item $w$ is a supersolution in $\Omega$,
		\item $w(x, t) >0$ for $(x, t)\in \Omega$,
		\item $\liminf_{(y, s) \to (x, t)}w(y, s) >0$ for $(x, t) \neq (0, 0) \in \partial \Omega$,
		\item $\lim_{(x, t) \to (0, 0)}w(x, t) =0.$
	\end{enumerate}
Our barrier will be on the form
\begin{equation*}
\label{eq:barrierform}
w(x, t)=f(t)e^{-\frac{|x|^2}{\beta t}}+ g(t),
\end{equation*}
for smooth functions $f$ and $g$. Differentiating formally, we get
\begin{equation}
\label{eq:timederivative}
w_t(x, t)=e^{-\frac{|x|^2}{\beta t}}\left(f'(t)+{\frac{|x|^2}{\beta t^2}}f(t)\right)+g'(t),
\end{equation}
\begin{equation}
\label{eq:gradient}
Dw(x, t)=-x\frac{2f(t)}{\beta t}e^{-\frac{|x|^2}{\beta t}},
\end{equation}
and
\begin{equation}
\label{eq:hessian}
\begin{split}
D^2w(x, t)&=
-\frac{2f(t)}{\beta t}e^{-\frac{|x|^2}{\beta t}}\mathbb{Id}_n+\frac{4f(t)}{t^2\beta^2}e^{-\frac{|x|^2}{\beta t}}x\otimes x \\
&=\frac{2f(t)}{\beta t}e^{-\frac{|x|^2}{\beta t}}\left(-\mathbb{Id}_n+\frac{2}{\beta t}x \otimes x\right).
\end{split}
\end{equation}
From \eqref{eq:hessian} we see that
\begin{equation}
\label{eq:trace}
\text{tr}(D^2w(x, t))=\frac{2f(t)}{\beta t}e^{-\frac{|x|^2}{\beta t}}\left(-n+\frac{2|x|^2}{\beta t}\right).
\end{equation}
From \eqref{eq:gradient} and \eqref{eq:hessian}, (or observing that $w(x, t)=G(r, t)$, and so $\Delta_{\infty}^Nw =G_{rr}$), we get
\begin{equation}
\label{eq:inftylaplace_p}
\left\langle D^2w\frac{Dw}{|Dw|}, \frac{Dw}{|Dw|} \right\rangle =\frac{2f(t)}{\beta t}e^{-\frac{|x|^2}{\beta t}}\left(-1+\frac{2|x|^2}{\beta t}\right).
\end{equation}
From \eqref{eq:inftylaplace_p} and \eqref{eq:trace}, we calculate
\begin{align*}
\mathcal{A}_pw =&\frac{1}{p}\text{tr}(D^2w) +\frac{p-2}{p}\left\langle D^2w\frac{Dw}{|Dw|}, \frac{Dw}{|Dw|} \right\rangle, \\
&=\frac{1}{p}\cdot\frac{2f(t)}{\beta t}e^{-\frac{|x|^2}{\beta t}}\left(-n+\frac{2|x|^2}{\beta t}\right) +\frac{p-2}{p}\cdot\frac{2f(t)}{\beta t}e^{-\frac{|x|^2}{\beta t}}\left(-1+\frac{2|x|^2}{\beta t}\right) \\
&=\frac{2f(t)}{\beta t}e^{-\frac{|x|^2}{\beta t}}\left(-\frac{n}{p}-\frac{p-2}{p}+\frac{2|x|^2}{\beta t}\left(\frac{1}{p}+\frac{p-2}{p}\right)\right)\\
&=\frac{f(t)}{\beta t}e^{-\frac{|x|^2}{\beta t}}\left(-2\left(\frac{n+p-2}{p}\right)+\frac{|x|^2}{\beta t}\left(4\frac{p-1}{p}\right)\right) \\
&=\frac{f(t)}{\beta t}e^{-\frac{|x|^2}{\beta t}}\left(-\alpha+\frac{|x|^2}{t}\right),
\end{align*}
where 
\[
\alpha =2\frac{n+p-2}{p}.
\]
This, together with \eqref{eq:timederivative}, gives
\begin{equation}
\label{eq:p-parabolicw}
\begin{split}
w_t-\mathcal{A}_pw \\
=&e^{-\frac{|x|^2}{\beta t}}\left(f'(t)+{\frac{|x|^2}{\beta t^2}}f(t)\right)+g'(t)-\frac{f(t)}{\beta t}e^{-\frac{|x|^2}{\beta t}}\left(-\alpha+\frac{|x|^2}{t}\right) \\
=&e^{-\frac{|x|^2}{\beta t}}\left(f'(t)+\frac{|x|^2f(t)}{\beta t^2}+\frac{\alpha f(t)}{\beta t}-\frac{|x|^2f(t)}{\beta t^2}\right) +g'(t) \\
=&e^{-\frac{|x|^2}{\beta t}}\left(f'(t)+\frac{\alpha f(t)}{\beta t} +g'(t)e^{\frac{|x|^2}{\beta t}}\right).
\end{split}
\end{equation}
Choose
\[
f(t)=-c\frac{1}{|\log|t||^{\delta+1}}, \ g(t)=\frac{1}{|\log|t||^{\delta}},
\]
for constants $0<c<1$, $\delta$ to be determined. We are now in position to prove the following theorem:
\begin{theorem}
	The smooth function $w:\Omega \to \mathbb{R}$ given by
	\begin{equation}
	\label{eq:the barrier}
	w(x, t)=-c\frac{1}{|\log|t||^{\delta+1}}e^{-\frac{|x|^2}{\beta t}} +\frac{1}{|\log|t||^{\delta}}
	\end{equation}
	is a barrier at $(0, 0)$.
\end{theorem}
\begin{proof}
	We check the requirements listed in Definition \ref{def:barrier}.
	
	{\bf 1.} We must check that $w$ is a viscosity supersolution in $\Omega$. Equation \eqref{eq:gradient} shows that $Dw=0$ precisely when $x=0$, so assume first that $x\neq 0$. It suffices to show that $w$ is a classical solution in this case. We first differentiate $f$ and $g$:
	\[
	f'(t)=-c(\delta+1)\frac{1}{t|\log|t||^{\delta+2}}, \ g'(t)=\delta\frac{1}{t|\log|t||^{\delta+1}}.
	\]
	Inserting the derivatives into \eqref{eq:p-parabolicw} gives
	\begin{align*}
	&w_t-\mathcal{A}_pw \\
	=&e^{-\frac{|x|^2}{\beta t}}\left(-c(\delta+1)\frac{1}{t|\log|t||^{\delta+2}}-c\frac{\alpha}{\beta}\frac{1}{t|\log|t||^{\delta+1}}+\delta\frac{1}{t|\log|t||^{\delta+1}}e^{\frac{|x|^2}{\beta t}}\right) \\
	=&\frac{1}{t|\log|t||^{\delta+1}}e^{-\frac{|x|^2}{\beta t}}\left(\frac{-c(\delta+1)}{|\log|t||}-c\frac{\alpha}{\beta}+\delta e^{\frac{|x|^2}{\beta t}}\right).
	\end{align*}
	$t$ is negative, so $e^{\frac{|x|^2}{\beta t}}<1$, hence
	\begin{align*}
	w_t-\mathcal{A}_pw >\frac{1}{t|\log|t||^{\delta+1}}e^{\frac{|x|^2}{\beta t}}\left(\frac{-c(\delta+1)}{|\log|t||}-c\frac{\alpha}{\beta}+\delta\right)
	\end{align*}
	For this to be positive, the expression inside the parentheses must be negative. Choosing 
	\begin{equation}
	\label{eq:choosing_delta}
	\delta =c\frac{\alpha}{\beta}
	\end{equation}
	ensures this, and with this choice $w$ is superparabolic in this case.
	
	Assume that $x=0$ so that $Dw=0$. From \eqref{eq:hessian} and \eqref{eq:trace} we deduce
	\[
	\text{tr}(D^2w(0, t))=c\frac{2n}{\beta t|\log|t||^{\delta+1}},
	\]
	and
	\[
	\lambda(D^2w(0, t))=c\frac{2}{\beta t|\log|t||^{\delta+1}}.
	\]
	Since
	\[
	w_t(0, t)=f'(t)+g'(t) =\frac{1}{t|\log|t||^{\delta+1}}\left(\frac{-c(\delta+1)}{|\log|t||}+\delta\right),
	\]
	Definition \ref{def:viscsuper} demands that we verify 
	\begin{align*}
	&\frac{1}{t|\log|t||^{\delta+1}}\left(\frac{-c(\delta+1)}{|\log|t||}+\delta\right) \\
	\geq &\frac{1}{p}\cdot c\frac{2n}{\beta t|\log|t||^{\delta+1}} +\frac{p-2}{p}\cdot c\frac{2}{\beta t|\log|t||^{\delta+1}}
	\end{align*}
	for $t<0$.
	This is equivalent to
	\[
	\frac{-c(\delta+1)}{|\log|t||}+\delta\leq \frac{2cn}{p\beta} +\frac{2c(p-2)}{p\beta} \\
	=2c\frac{p+n-2}{p\beta}=c\frac{\alpha}{\beta}.
	\]
	Because of our choice of $\delta$ in \eqref{eq:choosing_delta}, the above inequality is satisfied for all $t$, and \eqref{eq:the barrier} satisfies the second condition in Definition \eqref{def:viscsuper}. 
	
	It remains to show that $w$ is a viscosity supersolution. Let $\phi\in C^2(\Omega)$ touch $w$ from below at $(0, t)$. Since $w-\phi$ has a minimum at $(0, t)$, we have
	\[
	w_t=\phi_t, \ Dw=D\phi, \ D^2w>D^2\phi
	\]
	at this point. Since $D^2w(0, t)=\frac{2f(t)}{\beta t}\mathbb{Id}_n$, a scalar multiple of the identity matrix, this implies
	\[
	\lambda(D^2w((0, t))>\Lambda(D^2\phi(0, t))>\lambda(D^2\phi(0, t)),
	\]
	where $\lambda$ is the smallest eigenvalue and $\Lambda$ is the greatest. Since 
	\[
	\text{tr}(D^2\phi(0, t)) =\sum_{i=1}^{n}\lambda_i(D^2\phi(0, t)),
	\]
	we get
	\begin{align*}
	\phi_t(0, t)&=w_t(0, t)   \\
	&\geq\frac{1}{p}\text{tr}(D^2w(0, t))-\frac{p-2}{p}\lambda(D^2w(0, t)) \\
	&\geq\frac{1}{p}\text{tr}(D^2\phi(0, t))-\frac{p-2}{p}\lambda(D^2\phi(0, t)),
	\end{align*} 
	which implies that $w$ is indeed a viscosity supersolution.
	
	{\bf 2.} Since \eqref{eq:thedomain} implies $-\frac{|x|^2}{\beta t}<\log|\log|t||$, we see
	\[
	w(x, t)>-c\frac{1}{|\log|t||^{\delta+1}}e^{\log|\log|t||} +\frac{1}{|\log|t||^{\delta}} =\frac{1-c}{|\log|t||^{\delta}}>0,
	\]
	for $0<c<1$, as desired, and (2) in the Definition holds.
	
	{\bf 3.} $w$ is continuous in $\overline{\Omega}$, so we only need to check that the restriction of $w$ to $\partial \Omega$ is positive. We see
	\[
	\left.w(x, t)\right|_{\partial \Omega}=-c\frac{1}{|\log|t||^{\delta+1}}e^{\frac{-\beta t\log|\log|t||}{\beta t}} +\frac{1}{|\log|t||^{\delta}} =\frac{1-c}{|\log|t||^{\delta}}>0.
	\]
	
	{\bf 4.} We see that 
	\[
	\lim_{t \to 0^-}f(t)=\lim_{t \to 0^-}g(t)=0.
	\]
	Since $|x|^2<-\beta t\log|\log|t||\to 0$, we see
	\[
	e^{-\frac{|x|^2}{\beta t}} =\mathcal{O}(|\log|t||)
	\]
	as $t \to 0^-$. Therefore
	\[
	\lim_{(x, t)\to (0, 0^-)}w(x, t)=0,
	\]
	and (4) in the Definition is satisfied.
	
	Together, these points show that \eqref{eq:the barrier} is indeed a barrier at $(0, 0)$, and hence the origin is a regular point for the domain \eqref{eq:thedomain}.
\end{proof}
\begin{remark}
Since $\beta \to 4$ as $p\to \infty$, we see that \eqref{eq:thedomain} converges to the Petrovsky criterion for the $\infty$-parabolic equation \eqref{eq:infinitypetrovsky}. Note also that the result is completely independent of the number of spatial variables $n$.
\end{remark}
We now turn to the proof that Theorem \ref{thm:petrovsky} is sharp; any constant greater than $\beta$ in \eqref{eq:thedomain} will produce domain containting $\Omega$ where the origin is irregular.
\begin{theorem*}
 The origin is \emph{not} a regular point for the domain $\Omega$ enclosed by the hypersurfaces
 \begin{equation}
 \label{eq:irregulardomain}
 \begin{split}
 &\{(x, t)\in \mathbb{R}^n\times(-\infty, 0) \ : \ |x|^2 =-\beta(1+\epsilon)t\log|{\log{|t|}}|\} \\ &\text{and}\{t=-c \},
 \end{split}
 \end{equation}
for any $\epsilon >0$.
 \end{theorem*}
 \begin{proof}
 	The proof proceeds by constructing a domain $\tilde{\Omega}$ contained in $\Omega$, with the origin as common boundary point. We then show that $(0, 0)$ is irregular for $\tilde{\Omega}$, and Lemma \ref{cor:irregular} then implies that $(0, 0)$ regarded as a boundary point of $\Omega$ is irregular, too. 
  We shall construct a smooth function $w$ so that
 \begin{enumerate}
 	\item $w$ is subparabolic in $\tilde\Omega$,
 	%\item $v(0, 0) =0$.
 	\item $w$ is continuous on $\overline{\tilde\Omega}\setminus\{(0, 0)\}$,
 	\item The upper limit of $w$ at interior points converging to $(0, 0)$ is greater than its upper limit for the points converging to $(0,0)$ along the boundary.
 \end{enumerate}
 To see why the existence of such a $w$ implies that the origin is irregular, consider the boundary data $f: \partial \tilde\Omega \to \mathbb{R}$ defined as follows. Let $f=w$ near $(0, 0)$, and set 
 \[
 f(0, 0)=\lim_{\partial \tilde\Omega\ni(x, t)\to(0, 0)}v(x, t).
 \]
 As we shall see, this limit exists.
 For the rest of the boundary, continuously extend $f$ to a large constant $b$. 
 
 If $b$ is large enough, the comparison principle implies that every function $\overline{u}\in \mathcal{U}_f$ which satisfies $\overline{u}\geq f$ on $\partial \tilde\Omega$ also satisfies $\overline{u}\geq w$ in $\tilde\Omega$ since $w$ is a subsolution by (1). Taking the infimum over all such $\overline{u}$, we see $\overline{H}_f\geq w$ in $\tilde\Omega$, and hence by point (3) in the definition of $w$;
 \begin{align*}
 \limsup_{\tilde\Omega\ni(x, t) \to (0, 0)}\overline{H}_f(x, t) \geq &\limsup_{\tilde\Omega\ni(x, t) \to (0, 0)} w(x, t) \\ 
 >&\limsup_{\partial \tilde\Omega\ni(y, s) \to (0, 0)}w(y, s) =f(0, 0),
 \end{align*} 
 and so $(0, 0)$ is not a regular point for $\tilde \Omega$.
 
 Our function $w$ will be on the form
 \begin{equation}
 \label{eq:irregularitatsbarriare}
 w(x, t)=f(t)\mathrm{e}^{\frac{-|x|^2}{\beta t}k} +g(t),
 \end{equation}
 for suitable functions $f$ and $g$. Here $k\in (\frac{1}{2}, 1)$ will be chosen later, and $-1<t<0$. Indeed, we shall choose $t$ to be very close to 0. Calculating, we get
 \[
w_t(x, t ) = f'(t)\mathrm{e}^{\frac{-|x|^2}{\beta t}k}+\frac{f(t)|x|^2k}{\beta t^2}\mathrm{e}^{\frac{-|x|^2}{\beta t}k}+g'(t)
 \]
 and
\begin{align*}
& w_r = -f(t)\frac{2rk}{\beta t}\mathrm{e}^{\frac{-|x|^2}{\beta t}k}, \\
& w_{rr} = f(t)\mathrm{e}^{\frac{-|x|^2}{\beta t}k}\left(\frac{4r^2k^2}{\beta^2 t^2}-\frac{2k}{\beta t}\right)
\end{align*}
Inserting this into \eqref{eq:evoplap}, we get
\begin{equation}
\label{eq:choosefg}
w_t-\mathcal{A}_pw = \mathrm{e}^{\frac{-|x|^2}{\beta t}k}\left[f'(t)+f(t)\frac{|x|^2(k-k^2)}{\beta t^2}+f(t)\frac{\alpha k}{\beta t}\right]+g'(t),
\end{equation}
with $\alpha$ as in \eqref{eq:bestsolution}.
Choose 
\[
 f(t)=\frac{-1}{{|\log{|t|}|}^{1+\epsilon_1}} \ \text{and} \ g(t)= \frac{1}{\log|\log{|t|}|},  
\]
where $\epsilon_1$ is a positive constant.

{\bf The case $\bf{x\neq 0}$.} We see that $Dw=0$ precisely when $x=0$.  We show that \eqref{eq:irregularitatsbarriare} is a \emph{classical} subsolution when $x \neq 0$.

 Inserting derivatives into \eqref{eq:choosefg}, we get
\begin{equation}
\label{eq:bigstuff}
\begin{split}
w_t-\mathcal{A}_pw = \mathrm{e}^{\frac{-|x|^2}{\beta t}k}\left[\frac{-(1+\epsilon_1)}{t|\log{|t|}|^{2+\epsilon_1}}-\frac{|x|^2(k-k^2)}{\beta t^2|\log{|t|}|^{1+\epsilon_1}} - \right. \\\left. \frac{\alpha k}{\beta t|\log{|t|}|^{1+\epsilon_1}}+\mathrm{e}^{\frac{|x|^2}{\beta t}k}\frac{1}{t\cdot \log^2{|\log{|t|}|}\cdot|\log{|t|}|}\right].
\end{split}
\end{equation}
Multiplying by $t\cdot |\log{|t|}|^{1+\epsilon_1}<0$, we see that the sign of \eqref{eq:bigstuff} coincides with the sign of
\begin{align*}
-\frac{1+\epsilon_1}{\log{|t|}}+\frac{|x|^2}{\beta t}(k-k^2)+\frac{\alpha k}{\beta}-\mathrm{e}^{\frac{|x|^2}{\beta t}k}\frac{|\log{|t|}|^{\epsilon_1}}{\log^2{|\log{|t|}|}}.
\end{align*}
We can choose $|t|$ small enough that 
\[
\left|\frac{1+\epsilon_1}{\log{|t|}}\right|<\frac{\alpha k}{\beta},
\]
and then 
\begin{equation}
\label{eq:negative}
\frac{|x|^2}{\beta t}(k-k^2)+\frac{2\alpha k}{\beta}-\mathrm{e}^{\frac{|x|^2}{\beta t}k}\frac{|\log{|t|}|^{\epsilon_1}}{\log^2{|\log{|t|}|}} <0.
\end{equation}
This inequality is satisfied if $|x|$ is so small that 
\begin{equation}
\label{eq:x small}
\frac{2\alpha k}{\beta } <\mathrm{e}^{\frac{|x|^2}{\beta t}k}\frac{|\log{|t|}|^{\epsilon_1}}{\log^2{|\log{|t|}|}}.
\end{equation}
or if $|x|$ so large that 
\begin{equation}
\label{eq:x large}
\frac{|x|^2(k-k^2)}{\beta |t|}>\frac{2\alpha k}{\beta}.
\end{equation}
We argue that at least one of these inequalities must hold. Indeed, fix $|t|$ so that
\begin{equation}
\label{eq:t- condition}
\frac{\epsilon_1}{2}\log|\log{|t|}|>4\frac{\alpha}{\beta}.
\end{equation}

{\bf 1. } In the case \eqref{eq:x small}, we take logarithms to get
\[
 \log{\frac{2\alpha k}{\beta}} < \frac{|x|^2}{\beta t}k +\epsilon_1\log|\log{|t|}|-2\log\log|\log{|t|}|, 
\]
 or
 \begin{align*}
 \frac{|x|^2}{\beta|t|}k &<\epsilon_1\log|\log{|t|}|-2\log\log|\log{|t|}|-\log{\frac{2 \alpha k}{\beta}} \\
& <\epsilon_1\log|\log{|t|}|-\log{k} \\
&< \epsilon_1\log|\log{|t|}|-\frac{\epsilon_1}{2}\log|\log{|t|}|\\
 & =\frac{\epsilon_1}{2}\log|\log{|t|}|>2
 \end{align*}
 for $|t|$ small enough. Hence \eqref{eq:negative} is satisfied. 

{\bf 2.} On the other hand, if \eqref{eq:x large} holds,  we calculate
 \[
 \frac{|x|^2}{\beta |t|}>\frac{2\alpha}{\beta(1-k)}>4\frac{\alpha}{\beta}.
 \]
 Since we chose $|t|$ according to \eqref{eq:t- condition},  we have that at least one of the inequalities \eqref{eq:x small} or \eqref{eq:x large} is satisfied for any $x\neq 0$, and $w$ is a subsolution.

{\bf The case $\bf{x= 0}$.} Then $Dw=0$, and according to Definition \ref{def:viscsuper} we need to show that for every $\phi \in C^2(\Omega)$ so that $w-\phi$ has a maximum at $(0, t)$, we have
\begin{equation}
\label{eq:subparabolic}
\phi_t(0, t)\leq \frac{1}{p}\text{tr}(D^2\phi(0, t))+\frac{p-2}{p}\Lambda(D^2\phi(0, t)).
\end{equation}
We show that $w$ itself satisfies this condition. An argument similar to the one in the proof of Theorem \ref{thm:petrovsky} then shows that $w$ is a viscosity subsolution.

Inserting the derivatives at $(0, t)$, we see that \eqref{eq:subparabolic} reads
\[
f'(t)+g'(t) \leq -\frac{\alpha f(t)}{\beta t}k,
\]
or
\[
-\frac{1+\epsilon_1}{t\cdot|\log{|t|}|^{2+\epsilon_1}}  +\frac{1}{\log^2|\log{|t|}|\cdot|\log{|t|}|\cdot t} \leq \frac{\alpha k}{\beta t\cdot|\log{|t|}|^{1+\epsilon_1}}. 
\]
This is the same as
\[
\frac{1+\epsilon_1}{|\log{|t|}|}-\frac{|\log{|t|}|^{\epsilon_1}}{\log^2|\log{|t|}|}+\frac{\alpha k}{\beta}\leq 0,
\]
but this inequality is the same as the one in \eqref{eq:negative}, and because of our choices of $|t|$ and $k$. This shows that the condition \eqref{eq:subparabolic} holds, and $w$ is a subsolution even in this case.

 Now we consider the level set $w(x, t)=m$, $m<0$, and calculate
  \begin{align*}
  w(x, t)&=\frac{-1}{|\log{|t|}|^{\epsilon_1+1}}\mathrm{e}^{-\frac{|x|^2}{\beta t}k}+\frac{1}{\log|\log{|t|}|} =m \\
   &\iff \frac{-1}{|\log{|t|}|^{\epsilon_1+1}}\mathrm{e}^{-\frac{|x|^2}{\beta t}k} =m-\frac{1}{\log|\log{|t|}|} \\
 &\iff \mathrm{e}^{-\frac{|x|^2}{ \beta t}k} =|\log{|t|}|^{\epsilon_1+1}\left(\frac{1}{\log|\log{|t|}|}-m\right) \\
  &\iff -\frac{|x|^2}{\beta t}k = (\epsilon_1+1)\log|\log{|t|}|+\log\left(\frac{1}{\log|\log{|t|}|}-m\right), 
  \end{align*}
  or simply
  \begin{equation}
   \label{eq:smalldomain}
    x^2 =-\beta t\left(\frac{\epsilon_1+1}{k}\log|\log{|t|}|+\frac{1}{k}\log\left(\frac{1}{\log|\log{|t|}|}-m\right)\right).
   \end{equation}
   Letting $\tilde{\Omega}$ denote the domain enclosed by \eqref{eq:smalldomain} and the hyperplane $t=c<0$, we have that for $m<0$, the function $v$ \eqref{eq:irregularitatsbarriare} is negative in $\tilde{\Omega}$, and $w(x, 0)=0$. This shows that the origin is an irregular boundary point for $\tilde{\Omega}$.
   
    The inclusion $\tilde \Omega \subset \Omega$ requires that
   \[
   \frac{\epsilon_1+1}{k}\log|\log{|t|}|+\frac{1}{k}\log\left(\frac{1}{\log|\log{|t|}|}-m\right) <(1+\epsilon)\log|\log|t||
   \]
   for small $|t|$. Fix $k$ close to 1 and $\alpha$ close to 0 so that
   \[
   \frac{\epsilon_1+1}{k}<1+\frac{\epsilon}{2}.
   \]
   Thus we have to verify that
   \[
   \left(\frac{1}{\log|\log{|t|}|}+|m|\right)^{\frac{1}{k}} \leq |\log|t||^{\frac{\epsilon}{2}},
   \] 
   but this obviously holds for small $|t|$ since the left-hand side is bounded.
   
    Hence $\tilde{\Omega} \subset \Omega$ for $\epsilon_1$ and $c$ close to 0, and $k$ close to 1, $(0,0)$ is an irregular boundary point for $\Omega$ as well.
     \end{proof}
   \section*{Acknowledgement}
  \noindent The author would like to thank Jana Björn and Vesa Julin  for discovering a flaw in the original proof of the Petrovsky criterion.
\bibliographystyle{alpha}
\bibliography{/Users/nikubo74/Desktop/PhD/refs} 

\newcommand{\etalchar}[1]{$^{#1}$}
\begin{thebibliography}{{Ubo}17}

\bibitem[Bar52]{barenblatt}
Grigory~I. Barenblatt.
\newblock On self-similar motions of compressible fluid in a porous medium.
\newblock {\em Prikladnaya Matematika i Mekhanika (Applied Mathematics and
  Mechanics (PMM))}, 1952.

\bibitem[BBP17]{bjorn2017tusk}
Anders Bj{\"o}rn, Jana Bj{\"o}rn, and Mikko Parviainen.
\newblock The tusk condition and {P}etrovski criterion for the normalized
  $p$-parabolic equation.
\newblock {\em arXiv preprint arXiv:1712.06807}, 2017.

\bibitem[BG13]{banerjee2013gradient}
Agnid Banerjee and Nicola Garofalo.
\newblock Gradient bounds and monotonicity of the energy for some nonlinear
  singular diffusion equations.
\newblock {\em Indiana University Mathematics Journal}, pages 699--736, 2013.

\bibitem[BG15]{banerjee2015dirichlet}
Agnid Banerjee and Nicola Garofalo.
\newblock On the {D}irichlet boundary value problem for the normalized
  {$p$}-{L}aplacian evolution.
\newblock {\em Communications on Pure \& Applied Analysis}, 14(1), 2015.

\bibitem[BSA15]{imageenhancement}
George Baravdish, Olof Svensson, and Freddie Aastrom.
\newblock On backward $p(x)$-parabolic equations for image enhancement.
\newblock {\em Numerical Functional Analysis and Optimization}, 36(2):147--168,
  2015.

\bibitem[CL83]{crandall1983viscosity}
Michael~G. Crandall and Pierre-Louis Lions.
\newblock Viscosity solutions of {H}amilton-{J}acobi equations.
\newblock {\em Transactions of the American Mathematical Society},
  277(1):1--42, 1983.

\bibitem[CW03]{crandall2003another}
Michael~G. Crandall and Pei-Yong Wang.
\newblock Another way to say caloric.
\newblock {\em Journal of Evolution Equations}, 3(4):653--672, 2003.

\bibitem[DiB95]{dibenedetto1995}
Emmanuele DiBenedetto.
\newblock {\em Partial Differential Equations}.
\newblock Birkhäuser, 1995.

\bibitem[Doe09]{does2009evolution}
Kerstin Does.
\newblock An evolution equation involving the normalized {$p$}-{L}aplacian.
\newblock {\em Communications on Pure and Applied analysis}, 2009.

\bibitem[ES{\etalchar{+}}91]{evans1991motion}
Lawrence~C. Evans, Joel Spruck, et~al.
\newblock Motion of level sets by mean curvature. i.
\newblock {\em Journal of Differential Geometry}, 33(3):635--681, 1991.

\bibitem[JK06]{Juutinen2006}
Petri Juutinen and Bernd Kawohl.
\newblock On the evolution governed by the {I}nfinity {L}aplacian.
\newblock {\em Mathematische Annalen}, 335(4):819--851, 2006.

\bibitem[KL96]{kilpelainen1996dirichlet}
Tero Kilpel{\"a}inen and Peter Lindqvist.
\newblock On the {D}irichlet boundary value problem for a degenerate parabolic
  equation.
\newblock {\em SIAM Journal on Mathematical Analysis}, 27(3):661--683, 1996.

\bibitem[MPR10]{manfredi2010asymptotic}
Juan~J. Manfredi, Mikko Parviainen, and Julio~D. Rossi.
\newblock An asymptotic mean value characterization for a class of nonlinear
  parabolic equations related to tug-of-war games.
\newblock {\em SIAM Journal on Mathematical Analysis}, 42(5):2058--2081, 2010.

\bibitem[Per23]{perron1923neue}
Oskar Perron.
\newblock Eine neue {B}ehandlung der ersten {R}andwertaufgabe f{\"u}r {$\Delta
  u= 0$}.
\newblock {\em Mathematische Zeitschrift}, 18(1):42--54, 1923.

\bibitem[Pet35]{petrovsky1935ersten}
Ivan~G. Petrovsky.
\newblock Zur ersten {R}andwertaufgabe der {W}{\"a}rmeleitungsgleichung.
\newblock {\em Compositio Math}, 1:383--419, 1935.

\bibitem[PS{\etalchar{+}}08]{peres2008tug}
Yuval Peres, Scott Sheffield, et~al.
\newblock Tug-of-war with noise: A game-theoretic view of the $ p
  $-{L}aplacian.
\newblock {\em Duke Mathematical Journal}, 145(1):91--120, 2008.

\bibitem[Ste29]{sternberg1929gleichung}
Wolfgang Sternberg.
\newblock {\"U}ber die {G}leichung der {W}{\"a}rmeleitung.
\newblock {\em Mathematische Annalen}, 101(1):394--398, 1929.

\bibitem[{Ubo}17]{ubostad1}
N.~{Ubostad}.
\newblock {Boundary Regularity for the $\infty$-Heat Equation}.
\newblock {\em ArXiv e-prints}, September 2017.

\bibitem[Wat12]{watson2012introduction}
Neil~A. Watson.
\newblock {\em Introduction to Heat Potential Theory}, volume 182 of {\em
  Mathematical Surveys and Monographs}.
\newblock American Mathematical Society, 2012.

\end{thebibliography}

	\end{document}